\documentclass[10pt，oneside, reqno]{amsart}

\usepackage{xypic}
\input xy
\xyoption{all}
\usepackage{epsfig}
\usepackage{amsthm}
\usepackage{amssymb}
\usepackage{amsmath}
\usepackage{amscd}
\usepackage{color}
\usepackage{esint}
\usepackage{enumitem}
\usepackage[top=1.1in, bottom=1in, left=1in, right=1in]{geometry}

\usepackage[usenames,dvipsnames]{xcolor}
\usepackage[colorlinks=true, pdfstartview=FitV, linkcolor=blue, citecolor=blue, urlcolor=blue]{hyperref}

\newcommand{\lb}{\varLambda}

\def \<{\langle}
\def \>{\rangle}

\newcommand{\bg}{\begin{equation}}
\newcommand{\ed}{\end{equation}}
\newcommand{\bga}{\begin{eqnarray}}
\newcommand{\eda}{\end{eqnarray}}

\def\cbdu{\par{\raggedleft$\Box$\par}}

\newtheorem {Theorem}  {Theorem}

\numberwithin{Theorem}{section}

\newtheorem {Lemma}[Theorem]  {Lemma}
\newtheorem {Proposition}[Theorem]{Proposition}
\theoremstyle{definition}
\newtheorem{Definition}[Theorem]{Definition}
\theoremstyle{remark}

\def \l{\lambda}
\expandafter\chardef\csname pre amssym.def
at\endcsname=\the\catcode`\@ \catcode`\@=11
\def\undefine#1{\let#1\undefined}
\def\newsymbol#1#2#3#4#5{\let\next@\relax
 \ifnum#2=\@ne\let\next@\msafam@\else
 \ifnum#2=\tw@\let\next@\msbfam@\fi\fi
 \mathchardef#1="#3\next@#4#5}
\def\mathhexbox@#1#2#3{\relax
 \ifmmode\mathpalette{}{\m@th\mathchar"#1#2#3}%
 \else\leavevmode\hbox{$\m@th\mathchar"#1#2#3$}\fi}
\def\hexnumber@#1{\ifcase#1 0\or 1\or 2\or 3\or 4\or 5\or 6\or 7\or 8\or
 9\or A\or B\or C\or D\or E\or F\fi}

\font\teneufm=eufm10 \font\seveneufm=eufm7 \font\fiveeufm=eufm5
\newfam\eufmfam
\textfont\eufmfam=\teneufm \scriptfont\eufmfam=\seveneufm
\scriptscriptfont\eufmfam=\fiveeufm

\catcode`\@=\csname pre amssym.def at\endcsname

\newcounter{remark}
\newenvironment{remark}
{\medskip \stepcounter{remark} \noindent \textit{Remark
\arabic{section}.\arabic{remark}.}}{\rm \cbdu}

\def \onehalf {\frac{1}{2}}
\def \ddelta {\delta_{n+1}}

\newcommand{\supp}{{\mathit supp}\,}

\newcommand{\divv}{{\text {div}}\,}
\newcommand{\divergence}{\nabla \cdot}
\newcommand{\curl}{\nabla \times}

\renewcommand{\d}{\delta}
\newcommand{\om}{\omega}
\newcommand{\omp}{\omega^{(p)}}
\newcommand{\omc}{\omega^{(c)}}

\renewcommand{\l}{\lambda}
\newcommand{\Om}{\Omega}
\renewcommand{\a}{\alpha}

\renewcommand{\b}{\beta}

\newcommand{\gm}{\gamma}

\newcommand{\Gm}{\Gamma}
\newcommand{\R}{\mathbf{R}}

\newcommand{\les}{\lesssim}

\newcommand{\Dd}{{\mathcal D}}

\newcommand{\Rr}{{\mathcal R}}
\newcommand{\Ss}{{\mathcal S}}
\newcommand{\Tt}{{\mathcal T}}

\newcommand{\mB}{\overline{B}}
\newcommand{\mP}{\overline{p}}
\newcommand{\mR}{\overline{R}}
\newcommand{\tten}{\mathring{\otimes}}
\newcommand{\IdM}{\mathbf{Id}}
\newcommand\powermone[1]{#1^{-1}}
\newcommand\dist[1]{\text{dist}(#1)}

\def  \C   {{\mathbb C}}
\def  \R   {{\mathbb R}}
\def  \Z   {{\mathbb Z}}
\def  \N   {{\mathbb N}}

\def  \T   {{\mathbb T}}
\def  \W   {{\mathbb W}}
\def  \WW   {{\mathbf W}}
\def  \OOm  {{\mathbf \Omega}}

\def  \haf  {{\frac{1}{2}}}
\def  \p   {\partial}

\def  \eek   {\mathbf{e}_k}
\def  \npone {{n+1}}
\def  \mone  {{-1}}

\newcommand\onenorm[1]{\lVert#1\rVert_{L^1(\T^3)}}
\newcommand\twonorm[1]{\lVert#1\rVert_{L^2(\T^3)}}
\newcommand\rnorm[1]{\Vert#1\Vert_{L^r(\T^3)}}
\newcommand\Linfnorm[1]{\Vert#1\Vert_{L^\infty(\T^3)}}
\newcommand\Conenorm[1]{\Vert#1\Vert_{C^1(\T^3)}}
\newcommand\Ctwonorm[1]{\Vert#1\Vert_{C^2(\T^3)}}
\newcommand\Honenorm[1]{\Vert#1\Vert_{H^1(\T^3)}}
\newcommand\Hsnorm[1]{\Vert#1\Vert_{H^s(\T^3)}}
\newcommand\WNOnorm[1]{\Vert#1\Vert_{W^{N,1}(\T^3)}}

\newcommand\LinfNorm[1]{\left \lVert#1\right \rVert_{L^\infty(\T^3)}}

\newcommand\HoneNorm[1]{\left \lVert#1\right \rVert_{H^1(\T^3)}}

\def\build#1_#2^#3{\mathrel{\mathop{\kern 0pt#1}\limits_{#2}^{#3}}}

\begin{document}

\title[Stationary weak solutions of the EMHD equation]{Three dimensional stationary solutions of the Electron MHD equations}

\author [Qirui Peng]{Qirui Peng}
\address{Department of Mathematics, Stat. and Comp. Sci.,  University of Illinois Chicago, Chicago, IL 60607,USA}
\email{qpeng9@uic.edu}

\thanks{}





\begin{abstract} The goal of this paper is to construct non-trivial steady-state weak solutions of the three dimensional Electron Magnetohydrodynamics equations in the class of $H^s(\T^3)$ for some small $s > 0$. By exploiting the formulation of the stationary EMHD equations one can treat them as generalized Navier-Stokes equations with half Laplacian. Therefore with convex integration scheme we obtained such stationary weak solutions, which is not yet realizable in the case of classical 3D Navier-Stokes equations.

\bigskip


\hspace{0.02cm}CLASSIFICATION CODE: 35Q35, 76B03, 76W05 .\\

\end{abstract}

\maketitle

\section{Introduction
}
\subsection{The stationary EMHD equations}
The Electron Magnetohydrodynamics(EMHD) equations are given by
\begin{subequations}\label{EMHD}
\begin{align}
\partial_t B + \nabla \times \big ( (\nabla \times B) \times B \big)    &= \nu \Delta B, \label{EMHD 1}\\
                                                  \nabla \cdot B  &= 0, \label{EMHD 2}
\end{align}
\end{subequations}
where $B:[0,\infty) \times \T^3 \to \R^3$ is the unknown vector field. The number $\nu > 0$ stands for the resistivity constant. For convenience of notation in the following discussion we set $\nu = 1$. In this paper, our main interest is to construct weak solutions of the above equation, whose definition is given as follows.
\begin{Definition} Let $T > 0$ and denote the space of divergence free test function $\phi \in C^\infty_c ([0,T)\times \T^3)$ by $\Dd_T$. Given any weakly divergence free function $B_0 \in L^2(\T^3)$, we call $B \in L^2([0,T]\times \T^3)$ a weak solution to the equations $\eqref{EMHD}$ with initial data $B_0$ if $B$ is weakly divergence free for a.e. $t \in [0,T]$ and satisfies 
\begin{equation}\label{weak_formulation_EMHD}
\int_0^T \int_{\T^3} B \cdot \big(\p_t \phi + \Delta \phi + B \cdot \nabla (
\curl \phi)   \big) dxdt = -\int_{\T^3} B_0(x) \phi(x,0) dx,
\end{equation}
for all $\phi \in \Dd_T$, where we use the vector calculus identity 
\begin{equation}\label{vector_cal_identity}
 (\curl B) \times B  = (B \cdot \nabla) B  + \haf \nabla |B|^2.
\end{equation}
\end{Definition}
We wish to study the existence of stationary weak solution to $\eqref{EMHD}$, i.e. weak solutions that are independent of the time variable. By taking the inverse of curl on both side of $\eqref{EMHD 1}$, using $\eqref{vector_cal_identity}$ and realize that $\p_t B \equiv 0$, the stationary EMHD equations can be written as 
\begin{subequations}\label{stat_Hall_equation}
\begin{align}
\nabla \times B + (B\cdot \nabla) B + \nabla p &= 0, \label{stat_Hall_equation 1}\\
                                                  \divergence B   &= 0, \label{stat_Hall_equation 2}
\end{align}
\end{subequations}
where the scalar function $p(t,x)$ is usually known as the pressure. The expression for $p$ follows directly from $\eqref{vector_cal_identity}$ and can be solved by
\begin{equation}\label{formula_pressure}
p = \haf |B|^2 + \Delta^{-1} \divv \divv  (B\otimes B).
\end{equation}
Note the analogy between the equation $\eqref{stat_Hall_equation}$ and the stationary generalized Navier-Stokes equations 
\begin{subequations}\label{stat_NSE}
\begin{align}
\nu(-\Delta)^\haf u + (u\cdot \nabla) u + \nabla p &= 0, \label{stat_NSE}\\
                                                  \divergence u   &= 0, \label{stat_NSE}
\end{align}
\end{subequations}
\subsection{Main results and previous works} 
\begin{Theorem}\label{main_Theorem}
For any $s < \frac{1}{250}$, there exists a nontrivial stationary weak solution $B \in L^2(\T^3)$ of $(\ref{stat_Hall_equation})$ such that $B \in H^s(\T^3)$.
\end{Theorem}
\begin{remark}
Theorem $\ref{main_Theorem}$ also gives a non-uniqueness result for weak solutions of the equations $\eqref{EMHD}$.
\end{remark}
The EMHD has been studied extensively since the last decade. Here we summarize some recent works done in the context of EMHD equations:
\subsubsection{Well-posedness and ill-posedness}The local well-posedness of the classical solutions of $\eqref{EMHD}$ was initially proven by Chae, Degond, and Liu\cite{CDL14}. The most recent contiuation crition was established by Dai and Oh\cite{DO24}. The ill-posedness in the sense of the norm inflation was shown by Chae and Weng\cite{CW16} and lately by Jeong and Oh\cite{JO24}. The singularity formation of EMHD with $\nu = 0$ was constructed in the same paper of Chae, Weng.
In \cite{DGW24}, the authors established a sufficient condition under which a self-similar blow-up of the EMHD without resistivity cannot be developed.  
\subsubsection{Convex integration and stationary solutions} The method of convex integration scheme, also known as the Nash iteration scheme, has been developed comprehensively in the recent decade. At the moment, this machinery serves as one of the major tools in studying the pathological behavior of the fluid dynamics. The power of the scheme is that it allows us to construct weak solutions to different fluid models with arbitrary regularity within certain regime. Here we refer the readers to the survey paper \cite{BV19} for the development of the studies of the fluid equations via convex integration over the last decade. The most well-known triumph of the convex integration is perhaps the proof on the Onsager's conjecture of the Euler equations by Isett\cite{I18}. Later Novack and Vicol had accomplished another remarkable work on the $L^3-$ based Onsager's theorem\cite{NV23}. Very recently, the generalized Onsager's conjecture for the Surface quasi-geostrophic equations has also been proven indenpendently by Dai, Giri and Radu\cite{DGR24} and Looi and Isett\cite{LI24}. In regards to the EMHD, the author in \cite{D24} created non-unique weak solutions of EMHD in the class of $L^\gm_tW^{1,\infty}_x \cap L^1_t L^2_x$. The readers are also referred to \cite{LZZ24,D21} for the related works on the non-uniqueness of the MHD and Hall-MHD. On the topics of constructing nontrivial stationary weak solutions via convex integration, seemly the only work is thanks to Luo\cite{L19}, in which the author constructed stationary weak solutions for Navier Stokes equations in dimension $d \geq 4$.

\section{The iteration scheme and the main proposition} 
The weak solutions in Theorem $\ref{main_Theorem}$ will be obtained by the method of convex integration. To that end, we consider the stationary EMHD-Reynolds system which is given by 
\begin{subequations}\label{EMHD_Reynolds_system}
\begin{align}
 \nabla \times B + B\cdot \nabla B + \nabla p &= div R, \label{EMHD_Reynold_1}\\
                                                  \divv B   &= 0, \label{EMHD_Reynold_2}
\end{align}
\end{subequations}
where $R$ is a symmetric traceless $2-$tensor. The idea is to construct iteratively a sequence of solutions $B_n$ such that $(B_n,p_n,R_n)$ satisfies $\eqref{EMHD_Reynolds_system}$, i.e. 
\begin{subequations}\label{Iteration_system}
\begin{align}
 \nabla \times B_n + B_n\cdot \nabla B_n + \nabla p_n &= div R_n, \label{Iteration_system 1}\\
                                                  \divv B_n   &= 0, \label{Iteration_system 2}
\end{align}
\end{subequations}
where $n \in \N$. Moreover, $(B_n,R_n)$ satisfy the following inductive hypothesis: 
\begin{subequations}
\begin{align}
||R_n||_{L^1(\T^3)} \les \delta_{n+1}, \label{Inductive Hypothesis 1}\\
||B_n||_{L^2(\T^3)} \les 1-\delta^{\frac{1}{2}}_{n}, \label{Inductive Hypothesis 2}\\
||\nabla B_n||_{L^2(\T^3)} \les \l_n\delta^\haf_{n}, \label{Inductive Hypothesis 3}
\end{align}
\end{subequations}
where 
\begin{equation} \label{def_lambda_delta}
\lambda_n := \lceil{a^{b^n}} \rceil, \ \ \  \delta_n := \lambda^{-2\beta}_n
\end{equation}
with the values of $a, b$ and $\beta$ given by \eqref{def_b_beta}.

\begin{Proposition} \label{Main_proposition} With the choice of $a,b$ and $\beta$ in \eqref{def_b_beta}, we have that if the tuple $(B_n,p_n,R_n)$ for some $n \geq 1$ is a solution to the stationary EMHD-Reynolds system $\eqref{Iteration_system}$ which satisfies the inductive hypothesis $(\ref{Inductive Hypothesis 1}) - (\ref{Inductive Hypothesis 3})$, then there exists $(B_{n+1},p_{n+1},R_{n+1})$, a solution of the $\eqref{Iteration_system}$ satisfying the same induction hypothesis with $n$ replaced by $n+1$. Furthermore, the difference $B_{n+1}-B_n$ satisfies 
\begin{align}
\twonorm{B_{n+1}-B_n} &\lesssim \d^\haf_{n+1} \label{Main_prop_1},\\
\Honenorm{B_{n+1}-B_n} &\lesssim \l_{n+1} \d^\haf_{n+1}, \label{Main_prop_2}
\end{align}
where $\l_{n}$ and $\d_n$ are given by $\eqref{def_lambda_delta}$.
\end{Proposition}
\begin{proof}[Proof of Theorem \ref{main_Theorem}] Let $\beta = \frac{1}{250}$ and choose $a$ and $b$ as in proposition $\ref{Main_proposition}$. The tuple $(B_1,p_1,R_1)=(0,0,0)$ is a solution to $\eqref{Iteration_system}$, therefore by proposition $\ref{Main_proposition}$ there exists a sequence of solutions $\big \{(B_n,p_n,R_n)\big \}_{n\geq 1}$ such that they satisfy $\eqref{Inductive Hypothesis 1}-\eqref{Inductive Hypothesis 3}$ as well as $\eqref{Main_prop_1}-\eqref{Main_prop_2}$. Then for any $n\geq 1$, taking $0<s<\b$ yields 
\begin{align*}
\Hsnorm{B_n} &\leq \sum_{k=1}^{n-1} \Hsnorm{B_{k-1}-B_k} \\
&\lesssim \sum_{k=1}^{n-1} \Honenorm{B_{k+1}-B_k}^s \twonorm{B_{k+1}-B_k}^{1-s}\\
&\lesssim \sum_{k=1}^{n-1} \l^s_{k+1}\d_{k+1}^{\haf s} \d^{\haf(1-s)}_{k+1} = \sum_{k=1}^{n-1} \l^s_{k+1}\d^\haf_{k+1} =\sum_{k=1}^{n-1}\l^{s-\b}_{k+1} \\
&\lesssim \sum_{k=1}^\infty \l^{s-\b}_{k+1} < \infty.
\end{align*}
Therefore $B_n$ is uniformly bounded in $H^s$. By the Rellich compactness theorem, upon extracting a further subsequence if necessary, there exists a function $B$ such that $B_n \to B$ in $L^2(\T^3)$ and that $B \in H^s$. To show that $B$ is a weak solution to $\eqref{stat_Hall_equation}$, let $\phi \in \Dd_T$ and multiply $\curl \phi$ on both side of $\eqref{Iteration_system 1}$ we obtain
\[
\int_{T^3} (\curl \phi) \cdot \Big (\curl B_n + (B_n \cdot \nabla)B_n + \nabla p_n \Big)dx = \int_{\T^3} (\curl \phi) \divv R_n. 
\]
Use integration by part 
\[
\int_{T^3} B_n \cdot \Big (\Delta \phi + (B_n \cdot \nabla (\curl \phi)) \Big)dx - \int_{\T^3} \nabla(\curl \phi) R_n dx =0.
\]
Since by $\eqref{Inductive Hypothesis 1}$, $R_n \to 0$ in $L^1$ we have 
\[
\int_{\T^3} \nabla(\curl \phi) R_n \lesssim \onenorm{R_n}\Ctwonorm{\phi} \to 0.
\]
In addition, $B_n \to B$ in $L^2$ we have 
\[
\int_{T^3} B_n \cdot \Delta \phi dx \to \int_{T^3} B \cdot \Delta \phi dx
\]
and
\begin{align*}
&\bigg | \int_{\T^3} B_n \cdot \Big( (B_n \cdot \nabla) (\curl \phi) \Big) - \int_{\T^3} B \cdot \Big( (B \cdot \nabla) (\curl \phi) \Big) \bigg | \\
\leq &\bigg |\int_{\T^3} \Big(B_n-B \Big) \cdot \Big( (B_n \cdot \nabla) (\curl \phi) \Big)\bigg| + \bigg| \int_{\T^3} B \cdot \Big( \big((B_n-B) \cdot \nabla \big) (\curl \phi) \Big) \bigg | \\
\lesssim &\twonorm{B_n-B} \twonorm{B_n} \Ctwonorm{\phi} + \twonorm{B}\twonorm{B_n-B}\Ctwonorm{\phi} \to 0.
\end{align*}
The fact that $p_n \to p$ in $L^1$ follows from the Sobolev embedding $H^s \hookrightarrow L^q$, $q=\frac{6}{3-2s}$. Since $q > 2$ and Riesz operator is bounded from $L^{\frac{q}{2}}$ to $L^\frac{q}{2}(\T^3)$, we conclude that $p_n \to p$ in $L^\frac{q}{2}(\T^3)$, implying the convergence in $L^1(\T^3)$. 
\end{proof}

\section{proof of the Main Proposition}\label{proof_Main_Prop}
This section is devoted to the proof of proposition $\ref{Main_proposition}$. We start by choosing the suitable parameters
\subsection{Choice of parameters} We start from fixing several parameters, including those mentioned in $(\ref{def_lambda_delta})$.  
\begin{itemize}
\item[(1)] Let
\begin{equation}\label{def_b_beta}
a=5,\ \  b = 32, \ \ \beta = \frac{1}{250}.
\end{equation}
In view of the Theorem $(\ref{main_Theorem})$, our goal is to construct stationary weak solution $B$ such that $B \in H^s$ for any $s < \beta$. \\
\item[(2)] Define the spatial oscillation and concentration parameters $\sigma_{n+1}$  and $\mu_{n+1}$ respectively: 
\begin{equation}\label{def_sigma_mu}
\sigma_{n+1} = \l_{n+1}^\a,  \ \ \ \mu_{n+1} = \l_{n+1}^{1-\a},
\end{equation} \\
where $\a = \frac{16}{25}$. Note that $\l_{n+1} = \sigma_{n+1}\mu_{n+1}$.
\end{itemize}

\subsection{Mollification}\label{mollification}
Let $(B_n,p_n,R_n)$ be the tuple given by the $(\ref{Main_proposition})$, define the mollified solution $(\mB_l,\mP_l,\mR_l)$ by 
\begin{align}
\mB_l &= B_n \ast \eta_l, \label{mollified_B}\\
\mP_l &= p_n \ast \eta_l - \Big (|\mB_l|^2 - |B_n|^2 \ast \eta_l \Big ), \label{mollified_p}\\
\mR_l &= R_n \ast \eta_l + \mB_l \tten \mB_l - \big ( B_n \tten B_n \big ) \ast \eta_l, \label{mollified_R}
\end{align}
where $\tten$ is the traceless tensor product such that $a\tten b = a_i b_j - \d_{ij}a_jb_j$ and $\eta_l$ is the standard mollifier with the length scale $l$, i.e.
\begin{equation}\label{def_mollifier}
\eta_l (x) = l^{-3}\eta \big( \frac{x}{l}\big),  \qquad \int_{\T^3} \eta(x)dx = 1
\end{equation}
in which the length scale is given by 
\begin{equation} \label{molli_length_scale}
l^{-1} := \d^{-\frac{1}{2}}_{n+1} \l_n \d^{\frac{1}{2}}_n = \lambda^\gm_{n+1}
\end{equation}
and \begin{equation} \label{def_gamma}
\gm = \frac{1}{b}(1-\beta)+\beta = \frac{281}{8000} \approx 0.035.
\end{equation}
\begin{Lemma}\label{mollification_lemma} For any $m \in \N$, the mollification $(\mB_l,\mP_l,\mR_l)$ satisfy $\eqref{EMHD_Reynolds_system}$ and has the following estimate
\begin{subequations}
\begin{align}
||\mB_l - B_n||_{L^2(\T^3)} &\les \delta^{\frac{1}{2}}_{n+1}, \label{mollification_lemma1}\\ 
||\nabla^{m+1} \mB_l||_{L^2(\T^3)} &\les l^{-m} \lambda_{n}\delta^{\frac{1}{2}}_{n},  \label{mollification_lemma2}\\ 
||\nabla^m \mR_l ||_{L^1(\T^3)} &\les l^{-m} \delta_{n+1}. \label{mollification_lemma3}
\end{align}
\end{subequations}

\end{Lemma}
\begin{proof}
With straightforward computation we obtain 
\begin{align*}
\curl \mB_l + \divv \big( \mB_l \otimes \mB_l \big) + \nabla \mP_l &= \Big ( \curl B_n +\nabla p_n \Big)\ast \eta_l +\divv \big(\mB_l \otimes \mB_l \big) + \big (\nabla |\mB_l|^2 - \nabla |B_n|^2\ast \eta_l \big) \\
&= \divv R_n \ast \eta_l + \divv \big( \mB_l \tten \mB_l \big) - \divv \big( B_n \tten B_n \big) = \divv \mR_l.
\end{align*}
Notice that 
\begin{align*}
\mB_l - B_n = \int_{\T^3} \Big(B_n(y)-B_n(x) \Big)\eta_l(x-y) dy = \int_{\T^3} \Big(B_n(x-lz)-B_n(x) \Big)\eta(z) dz \les l |\nabla B_n|,
\end{align*}
hence 
\[
\twonorm{\mB_l - B_n} \les l \twonorm{\nabla B_n} \les (\l_n\d^\haf_n)^{-1} \d^\haf_{n+1} (\l_n\d^\haf_n) = \d^\haf_{n+1}.
\]
This establishes $\eqref{mollification_lemma1}$. $\eqref{mollification_lemma2}$ is obtained by an integration by part with a simple estimate :
\[
\twonorm{\nabla^{m+1}\mB_l} \les l^{-m}\twonorm{\nabla B_n} \|\eta \|_{C^m(\T^3)}\les l^{-m} \l_n \d^\haf_n.
\]
From $\eqref{mollified_R}$ we have that 
\begin{align*}
\onenorm{\nabla^m R_l} &\les \onenorm{\nabla^m (R_n \ast \eta_l)} + \onenorm{\nabla^m \Big( \mB_l \tten \mB_l - (B_n \tten B_n)\ast \eta_l \Big)}\\
&\les l^{-m}\d_{n+1} + l^{2-m} \twonorm{\nabla B_n}^2 \les l^{-m}\d_{n+1} + l^{-m} (\l_n \d^\haf_n)^{-2} \d_{n+1} (\l_n \d_n^\haf)^2 \\
&\les l^{-m} \d_{n+1},
\end{align*}
where we used the lemma $\ref{lemma_CET_commutator_est}$. 
\end{proof}

\subsection{Stationary Mikado flow} In this paper, Mikado flow introduced by \cite{DS17} as the building blocks is used for the convex integration scheme. To begin with, we present a version of geometric lemma due to the works by Nash \cite{Nas54}. Denote $\mathcal{S}^{d \times d}_+$ be the set of positive definite symmetrical $d\times d$ matrices and $\mathbf{e}_k = \frac{k}{|k|}$.

\begin{Lemma}\label{geometric_lemma} For any compact subset $\mathcal{N} \subset \mathcal{S}^{d \times d}_+$, there exists a finite set $\lb \subset \Z^d$ and smooth function $\Gm_k \in C^\infty(\mathcal{N};\R)$ for any $k\in \lb$ such that 
\begin{equation*}
R = \sum_{k\in \lb} \Gm^2_k (R) \mathbf{e}_k \otimes \mathbf{e}_k, \ \ \ \textit{for all } R \in \mathcal{N}.
\end{equation*}
\end{Lemma}
\begin{proof}
See \cite{DLS13}, lemma 3.2.
\end{proof}
The construction of the stationary Mikado flow can be found in section 4.1 of \cite{CL22} and section 3.2 of \cite{L19}. We provide a summary for the readers convenience. Let $\mathcal{N} = B_\haf (\IdM)$ where $B_\haf (\IdM)$ stands for a ball of radius $\haf$ centered at the identity matrix in the space $\mathcal{S}^{3 \times 3}_+$. From the above lemma we obtain a finite subset $\Lambda$ of integers in $\Z^3$ and $\Gm_k \in C^\infty(B_\haf (\IdM);\R)$ for each $k \in \Lambda$. Now choose $p_k \in (0,1)^3$ such that $p_k \neq p_{-k} + sk$ for any $s\in \R$ if both $k,-k \in \Lambda$. For each $k \in \Lambda$, let $m_k$ be the periodic line passing through $p_k$:
\begin{equation}\label{def_periodic_line_pk}
m_k := \big \{ sk+p_k \in \T^3 : s \in \R \big \}.
\end{equation}
From the choice of $p_k$ we know that $m_k \cap m_{-k} = \varnothing$ and there exist a constant $\kappa_\Lambda > 0$ such that
\[
|m_k \cap m_{k'}| \leq \kappa_\Lambda, \quad \forall k,k' \in \Lambda
\]
because $\Lambda$ is finite. Recalling from $\eqref{def_sigma_mu}$ that $\mu_{n+1} = \l^{1-\a}_{n+1}$, let $\phi, \Phi \in C^\infty_c \big([\haf, 1] \big)$ be such that 
\begin{equation}\label{condition_phi_Phi}
\Delta \Phi^{n+1}_k = \phi^{n+1}_k  \ \ \text{on} \ \  \T^3 \qquad \text{and} \qquad \int_{\T^3} (\phi^{n+1}_k)^2 = 1,
\end{equation}
where $\Phi^{n+1}_k$ and $\phi^{n+1}_k$ are given by 
\begin{equation}\label{def_Phi}
\Phi^{n+1}_k := \mu^{-1}_{n+1} \Phi \big(\mu_{n+1} \dist{x,m_k} \big) 
\end{equation}
and
\begin{equation}\label{def_phi}
\phi^{n+1}_k := \mu_{n+1} \phi \big(\mu_{n+1} \dist{x,m_k} \big). 
\end{equation}
Note that there exist $K_\Lambda > 0$ such that 
\begin{equation}\label{supp_pipe_intersection}
\supp \phi^n_k \cap \supp \phi^n_{k'} \subseteq \big \{x \in \T^3: \dist{x,m_k\cap m_{k'}} \leq K_\Lambda \mu^{-1}_{n+1}  \big \}
\end{equation}
for $k \neq k'$. The stationary Mikado flow $\WW^{n+1}_k$ is then defined by
\begin{equation}\label{def_stat_Mikado}
\WW_k^{n+1} = \phi^{n+1}_k \mathbf{e}_k. 
\end{equation}
Notice that $\WW_k^{n+1}$ can also be written by a skew-symmetric tensor $\OOm^{n+1}_k$: 
\begin{equation}\label{def_Omegak}
\OOm^{n+1}_k = \eek \otimes \nabla \Phi^{n+1}_k - \nabla \Phi^{n+1}_k \otimes \eek 
\end{equation}
since 
\begin{align}
\big ( \divv \OOm^{n+1}_k \big)_i &= \p_j [(\eek)_i (\nabla \Phi^{n+1}_k)_j] - \p_j [(\eek)_j (\nabla \Phi^{n+1}_k)_i] \notag \\
&= (\eek)_i \p_j(\nabla \Phi^{n+1}_k)_j = (\eek)_i \Delta \Phi^{n+1}_k = \phi^{n+1}_k (\eek)_i = (\WW^{n+1}_k)_i \label{equ_Omega_and_W}.
\end{align}
\\
\begin{remark} To be precise, the superscript of the Mikado flow $\WW^\npone_k$ should be written as $\WW^{\mu_\npone}_k$, which stands for the pipe flow with thickness $\mu^\mone_\npone$. Later we will use Mikado flows with spatial oscillation $\sigma_{n+1}$ and concentration $\mu_{n+1}$. They will be written as 
\[
\WW^\npone_k (\sigma_\npone x) \ \ \text{instead of}  \ \ \WW^{\mu_\npone}_k (\sigma_\npone x)
\]

\end{remark}
Next we introduce several important properties of the stationary Mikado flows by the following theorem.
\begin{Theorem}\cite{CL22}\label{prop_Mikado} For dimension $d = 3$, the stationary Mikado flows $\WW : \T^3 \to \R^3$ satisfy the following: 
\begin{enumerate}[label=\itshape(\roman*)]
\item\label{prop_Mikado1} Each $\WW^{n+1}_k \in C^\infty_c(\T^d)$ is divergence-free and 
\begin{equation*}
\WW^{n+1}_k = \divv \OOm^{n+1}_k.
\end{equation*}
In addition, it is a solution to the pressureless Euler equations
\begin{equation*}
\divv \Big( \WW^{n+1}_k \otimes \WW^{n+1}_k \Big) = 0.
\end{equation*}
\item\label{prop_Mikado2} For any $1 \leq p \leq \infty$, $N \in \N$, we have the estimates \\
\begin{align*}
\mu_n^{-N} ||\nabla^N \WW^{n+1}_k||_{L^p(\T^3)} &\les \mu_{n+1}^{1-\frac{2}{p}}, \\
\mu_n^{-N} ||\nabla^N \OOm^{n+1}_k||_{L^p(\T^3)} &\les \mu_{n+1}^{-\frac{2}{p}},
\end{align*}\\
which hold uniformly in $\mu_{n+1}$. \\
\item\label{prop_Mikado3} For any $k\in \lb$, one has 
\begin{equation*}
\fint_{\T^d} \WW^{n+1}_k \otimes \WW^{n+1}_k = \mathbf{e}_k \otimes \mathbf{e}_k,
\end{equation*}\\
moreover for any $1 \leq p \leq \infty$,\\
\begin{equation*}
||\WW^{n+1}_k \otimes \WW^{n+1}_{k'}||_{L^p(\T^d)} \les \mu_{n+1}^{2-\frac{3}{p}} \ \ \ \textit{for } k \neq k'.
\end{equation*}
\end{enumerate}
\end{Theorem}
\begin{proof}
Since $\phi^{n+1}_k$ is a smooth function whose level set is concentric cylinder with axis $k$, we have $\nabla \phi^{n+1}_k \cdot k = 0$. This implies that $\WW^{n+1}_k$ is divergence-free, as 
\[
\divv \WW^{n+1}_k = \nabla \phi^{n+1}_k \cdot k = 0.
\]
Moreover, we have 
\[
\bigg [ \divv \Big( \WW^{n+1}_k \otimes \WW^{n+1}_k \Big)  \bigg]_i = (\nabla \phi^{n+1}_k \cdot k ) \phi^{n+1}_k (\eek)_i = 0.
 \]
These together with $\eqref{equ_Omega_and_W}$ proves \ref{prop_Mikado1}. To show \ref{prop_Mikado2}, consider 
\begin{align*}
\|\nabla^N \WW^{n+1}_k \|_{L^p(\T^3)} \les \mu^{1+N}_{n+1} \bigg ( \int_{\supp \phi^{n+1}_k} \bigg )^\frac{1}{p} \les \mu^{1+N}_{n+1} |\supp \phi^{n+1}_k|^\frac{1}{p} = \mu^{1+N}_{n+1} \mu_{n+1}^{-\frac{2}{p}}
\end{align*}
and the estimate for $\OOm^n_k$ can be deduced in a similar manner. The identity 
\begin{equation*}
\fint_{\T^d} \WW^{n+1}_k \otimes \WW^{n+1}_k = \mathbf{e}_k \otimes \mathbf{e}_k
\end{equation*}
follows directly from the definition $\eqref{condition_phi_Phi}$ and $\eqref{def_stat_Mikado}$. If $k \neq k'$, then 
\[
||\WW^{n+1}_k \otimes \WW^{n+1}_{k'}||_{L^p(\T^d)} \les \mu_n^2 |\supp \phi^\npone_k \cap \supp \phi^\npone_{k'}|^\frac{1}{p} \les \mu^{2-\frac{3}{p}}_\npone 
\]
since by $\eqref{supp_pipe_intersection}$ we can cover the set $\supp \phi^\npone_k \cap \supp \phi^\npone_{k'}$ by finitely many balls of radius of the scale $\mu^{-1}_\npone$. These conclude the proof of \ref{prop_Mikado3}.
\end{proof}
\vspace{0.5cm}
\subsection{The magnetic field perturbation} Our goal is to choose the candidate of the solution $B_{n+1}$ based on the current one $\mB_l$ 
\begin{equation}\label{New_Magnetic_field}
B_{n+1} = \mB_l + \om_{n+1}.
\end{equation}
We call $\om_{n+1}$ the perturbation function, which consists of two components:
\begin{equation}\notag
\om_{n+1} = \om^{(p)}_{n+1} + \om^{(c)}_{n+1}.
\end{equation}
We call $\omp_{n+1}$ the principle part of the perturbation, which consists of a sum of $\WW^{n+1}_k$ with spatial oscillation $\sigma_{n+1}$ and concentration $\mu_{n+1}$: 
\begin{equation}\label{def_principle part}
\omp_{n+1} = \sum_{k \in \lb} a_k(x) \WW^{n+1}_k(\sigma_{n+1}x).
\end{equation}
Here $a_k(x)$ represents the amplitude functions which will be introduced in the next section. The role of the corrector part of the perturbation $\omc_{n+1}$ is to ensure that $B_{n+1}$ is divergence free: 
\begin{equation}\label{def_corrector part}
\omc_{n+1} = \sigma^{-1}_{n+1}\sum_{k \in \lb} \nabla a_k(x) :\OOm^{n+1}_k(\sigma_{n+1}x).
\end{equation}
\subsection{Amplitude functions}
Let $\chi$ be a monotone increasing smooth function \\
\begin{align*}
\chi(x) := \begin{cases}
4 \Big (||\mR_l||_{L^1(\T^3)} + \d_{n+1} \Big ),  \ \  &\textit{if} \quad 0 \leq |x| \leq ||\mR_l||_{L^1(\T^3)},\\\\
4|x|, \ \  &\textit{if} \quad  |x| \geq 2||\mR_l||_{L^1(\T^3)}.\\
\end{cases}
\end{align*}\\
In addition, define the cut-off function $\rho(x) \in C^\infty (\T^3)$ by \\ 
\begin{equation}\label{def_cutoff}
\rho(x) := \chi(\mR_l). 
\end{equation}\\
The amplitude function $a_k: \T^3 \to \R$ is then given by \\
\begin{equation}\label{def_amplitude}
a_k(x) = \rho^\frac{1}{2} \Gm_k \Big(\IdM - \frac{\mR_l}{\rho} \Big),
\end{equation}
in which $\Gm_k$ are smooth functions appear in lemma $\ref{geometric_lemma}$. 
\begin{Lemma}\label{lemma_Reducing_stress} The amplitude functions $a_k$ satisfy
\begin{equation*}
\sum_{k \in \lb} a^2_k \fint_{\T^3} \WW^{n+1}_k \otimes \WW^{n+1}_k dx = \rho \IdM - \mR_l.
\end{equation*}
\end{Lemma}
\begin{proof}
By Theorem $\ref{prop_Mikado}$, \ref{prop_Mikado3} and the geometric lemma $\ref{geometric_lemma}$ we deduce that
\[
\sum_{k \in \lb} a^2_k \fint_{\T^3} \WW^{n+1}_k \otimes \WW^{n+1}_k dx = \sum_{k \in \lb} \rho \Gm^2_k \Big(\IdM - \frac{\mR_l}{\rho} \Big)  \eek \otimes \eek =  \rho \IdM - \mR_l,
\]
due to the fact that $\IdM - \frac{\mR_l}{\rho} \in B_\haf (\IdM)$. 
\end{proof}
\vspace{0.5cm}
\subsection{The new Reynolds stress} With $B_{n+1}$ in our hand, we then aim to find the corresponding Reynolds stress $R_{n+1}$ as well as the pressure $p_{n+1}$ such that $(B_{n+1},p_{n+1},R_{n+1})$ solve the system $\eqref{Iteration_system}$. To this end, we define $R_{n+1}$ by \\
\begin{equation}\label{New_Reynold_stress}
R_{n+1} := R_{\textit{osc}} + R_F+ R_C+R_L. 
\end{equation}\\
The first term $R_{\textit{osc}}$ stands for the high oscillatory part of the new Reynolds stress and is given by \\
\begin{equation}\label{R_oscx}
R_{\textit{osc}} = \sum_{k \in \lb} \Tt \Big( \nabla(a^2_k), \WW^\npone_k(\sigma_{n+1}x)\otimes \WW^\npone_k(\sigma_{n+1}x) - \fint_{\T^3} \WW^\npone_k \otimes \WW^\npone_k \Big),
\end{equation}\\
where $\Tt : \C^\infty(\T^3,\R^3) \times C^\infty(\T^3,\R^{3\times 3}) \to C^\infty(T^3, S^{3 \times 3}_0)$ is known as the Bilinear anti-divergence: 
\begin{equation}\label{Bilinear_antidivergence}
\Tt (u,H) := u \Rr H - \Rr \big ( \nabla u \Rr H \big),
\end{equation}
in which $\Rr$ is the standard antidivergence operator (See Appendix \ref{appendix2}). \\\\
The next term $R_F$ consists of the error introduced by interference between different high frequency components: \\
\begin{equation}\label{R_F}
R_F = \sum_{k \neq k'} a_k a_{k'} \WW^\npone_k(\sigma_{n+1}x) \otimes \WW^\npone_k(\sigma_{n+1}x).
\end{equation}\\
The third term $R_C$ is due to the error introduced by the corrector function $\omc_{n+1}$: \\
\begin{equation}\label{R_C}
R_C = \Rr \Big ( \divv \big( \omc_{n+1} \otimes \om_{n+1} + \omp_{n+1} \otimes \omc_{n+1} \big )  \Big).  
\end{equation}\\
The last term $R_L$ is known as the linear error:\\
\begin{equation}\label{R_L}
R_L = \Rr \Big( \nabla \times \om_{n+1} + \divv \big( \mB_l \otimes \om_{n+1} + \om_{n+1} \otimes \mB_l \big) \Big). 
\end{equation}\\
Additionally we choose the new pressure term $p_{n+1}$ as \\
\begin{equation}\label{New_pressure}
p_{n+1} := \mP_l - \rho. 
\end{equation}
\begin{Lemma}\label{Next_iteration}
The tuple $(B_{n+1},p_{n+1},R_{n+1})$ defined by $(\ref{New_Magnetic_field})$, $(\ref{New_pressure})$ and $(\ref{New_Reynold_stress})$ respectively is a solution to the stationary EMHD-Reynolds system $\eqref{Iteration_system}$.
\end{Lemma}
\begin{proof} 
By direction computation, since $(\mB_l,\mP_l,\mR_l)$ solves $\eqref{Iteration_system}$ we have 
\begin{align*}
&\curl B_\npone + \divv \big( B_\npone \otimes B_\npone \big) + \nabla p_\npone = \curl \big( \om_\npone + \mB_l \big) + \divv \big( B_\npone \otimes B_\npone \big) + \nabla \mP_l - \nabla \rho \\\\
=& \curl \om_\npone + \curl \mB_l + \divv \big(\om_\npone \otimes \om_\npone \big) + \divv \big(\om_\npone \otimes \mB_l + \mB_l \otimes \om_\npone  \big) + \divv \big(\mB_l \otimes \mB_l \big) + \nabla \mP_l - \nabla \rho \\\\
=& \divv \mR_l + \curl \om_\npone + \divv \big( \om_\npone \otimes \mB_l+\mB_l \otimes \om_\npone \big) + \divv \big( \om_\npone \otimes \om_\npone \big) -\nabla \rho \\\\
=& \divv \mR_l + \divv R_L + \divv \big( \omp_\npone \otimes \omc_\npone \big) + \divv \big( \omc_\npone \otimes \om_\npone \big) + \divv \big( \omp_\npone \otimes \omp_\npone \big) - \nabla \rho \\\\
=& \divv \big( \omp_\npone \otimes \omp_\npone \big)-\nabla \rho  + \divv \mR_l + \divv R_C + \divv R_L.
\end{align*}\\
Recall from $\eqref{def_principle part}$ the expression of $\omp_\npone$, we then obtain
\begin{align*}
&\divv \big( \omp_\npone \otimes \omp_\npone \big)-\nabla \rho + \divv \mR_l-\divv R_F = \divv \Big( \sum_{k\in \Lambda} a^2_k \big ( \WW^\npone_k (\sigma_\npone x) \otimes \WW^\npone_k (\sigma_\npone x) \big)  + \mR_l \Big) - \nabla \rho \\\\
= &\divv \Big( \sum_{k\in \Lambda} a^2_k \big ( \WW^\npone_k (\sigma_\npone x) \otimes \WW^\npone_k (\sigma_\npone x) - \fint_{\T^3} \WW^{n+1}_k \otimes \WW^{n+1}_k   +\fint_{\T^3} \WW^{n+1}_k \otimes \WW^{n+1}_k \big)  + \mR_l \Big) - \nabla \rho \\\\
= &\divv \Big( \sum_{k\in \Lambda} a^2_k \big ( \WW^\npone_k (\sigma_\npone x) \otimes \WW^\npone_k (\sigma_\npone x) - \fint_{\T^3} \WW^{n+1}_k \otimes \WW^{n+1}_k \big) \Big) + \nabla \rho - \divv \mR_l + \divv \mR_l - \nabla \rho \\\\
= & \sum_{k\in \Lambda} \nabla (a^2_k) \big ( \WW^\npone_k (\sigma_\npone x) \otimes \WW^\npone_k (\sigma_\npone x) - \fint_{\T^3} \WW^{n+1}_k \otimes \WW^{n+1}_k \big) = \divv R_{\textit{osc}}
\end{align*}
and therefore
\[
\divv \big( \omp_\npone \otimes \omp_\npone \big)-\nabla \rho + \divv \mR_l = \divv R_{\textit{osc}} + \divv R_F.
\]\\
\end{proof}
\subsection{Estimation} In view of the lemma $\ref{Next_iteration}$, it suffices to show that $(B_{n+1},p_{n+1},R_{n+1})$ satisfies the inductive hypothesis $(\ref{Inductive Hypothesis 1}) - (\ref{Inductive Hypothesis 3})$ with $n$ replaced by $n+1$. In this section we will accomplish this by breaking it down into many estimates. First of all, let us give several relevant estimates on the amplitude function.\\
\begin{Lemma}\label{lemma_Amplitude_estimation}
The amplitude function $a_k$ defined in $(\ref{def_amplitude})$ satisfies the following estimates: \\
\begin{enumerate}[label=\itshape(\roman*)]
\item\label{amp_L2} The $L^2$ estimate 
\begin{equation*}
||a_k||_{L^2(\T^3)} \les \d^\frac{1}{2}_{n+1}.
\end{equation*}\\
\item\label{amp_WN1} For any positive integer $N$, one has the estimate in Sobolev space $W^{N,1}$ of $a^2_k$
\begin{align*}
||a_k||_{W^{N,1}} &\les l^{-4-N} \d^\haf_{n+1},\\
||a^2_k||_{W^{N,1}} &\les l^{-4-N} \d_{n+1}. 
\end{align*}\\

\item\label{amp_CN} For any non-negative integer $N$, one has the $N^{th}$ derivative estimate
\begin{align*}
||a_k||_{C^N} &\les l^{-8-N} \d^\frac{1}{2}_{n+1}, \\
||a^2_k||_{C^N} &\les l^{-8-N} \d_{n+1}.\\
\end{align*}
\end{enumerate}
\end{Lemma}

\begin{proof}
Recall that 
\begin{equation*}
a_k(x) = \rho^\frac{1}{2} \Gm_k \Big(\IdM - \frac{\mR_l}{\rho} \Big)
\end{equation*}
hence \ref{amp_L2} simply follows from 
\[
\twonorm{a_k} \les \onenorm{\mR_l}^\haf \les \d_\npone^\haf
\]
For \ref{amp_WN1} we use lemma $\ref{lemma_holder_composition}$ and Sobolev embedding $W^{4,1} \hookrightarrow L^\infty$ to get that 
\begin{align*}
\WNOnorm{a^2_k} &\les \sum_{i=0}^N \onenorm{\nabla^i \mR_l} \LinfNorm{\nabla^{N-i}\Gm^2_k\big( \IdM - \frac{\mR_l}{\rho} \big) } \\
&\les \sum_{i=0}^N \onenorm{\nabla^i \mR_l} \bigg ( \LinfNorm{\nabla^{N-i}\Big( \frac{\mR_l}{\rho}\Big)} \sum_{j=1}^{N-i} \Linfnorm{\nabla^j \Gm^2_k} \LinfNorm{\frac{\mR_l}{\rho}}^{j-1} \bigg)\\
&\les \sum_{i=0}^N \onenorm{\nabla^i \mR_l}  \Linfnorm{\nabla^{N-i}\mR_l}\d^\mone_\npone  \\
&\les \sum_{i=0}^N \onenorm{\nabla^i \mR_l}  \|\mR_l\|_{W^{4+N-i,1}}\d^\mone_\npone \les  \sum_{i=0}^N l^{-i}\d_\npone l^{-(4+N-i)} \d_\npone \d^\mone_\npone \\
&\les l^{-4-N} \d_\npone 
\end{align*}
Similarly we obtain that 
\[
\WNOnorm{a_k} \les l^{-4-N} \d^\haf_\npone. 
\]\\
Finally, \ref{amp_CN} follows from \ref{amp_WN1} and the Sobolev embedding $W^{4,1} \hookrightarrow L^\infty$. \\
\end{proof}

\subsubsection{Estimate the perturbation $\om_{n+1}$} We begin with estimating the perturbation function here. By the improved H\"older's inequality (see lemma \ref{lemma_immproved_holder}) as well as Theorem $\ref{prop_Mikado}$, \ref{prop_Mikado2}, we first obtain an $L^2$ estimate on the principle part. Recall from $\eqref{def_principle part}$ that 
\[
\omp_{n+1} = \sum_{k \in \lb} a_k(x) \WW^\npone_k(\sigma_{n+1}x),
\]
we have 
\begin{align}\label{L2_principle}
\twonorm{\omp_{n+1}} &\les ||a_k||_{L^2(\T^3)} ||\WW^\npone_k||_{L^2(\T^3)} + \sigma^{-\onehalf}_{n+1} \Conenorm{a_k} ||\WW^\npone_k||_{L^2(\T^3)} \notag \\ \notag \\
& \les \d^\onehalf_{n+1} + \sigma^{-\onehalf}_{n+1} l^{-9} \d^\onehalf_{n+1} \les \d^\onehalf_{n+1}, 
\end{align}
provided that $\sigma^{-\onehalf}_{n+1}l^{-9} \les 1$, or equivalently 
\begin{align*}
\sigma^{-\onehalf}_{n+1}l^{-9} \les 1 \Leftrightarrow \l^{-\onehalf \alpha +9\gm}_{n+1} \les 1  \Leftrightarrow  -\haf \alpha + 9\gamma \leq  0.
\end{align*}\\
This is guaranteed by our choice of parameters, indeed, recall from $(\ref{def_gamma})$ that $\gm = \frac{281}{8000}$, 
\[
 -\haf \alpha + 9\gamma = -\haf \times \frac{16}{25}+ 9 \times \frac{281}{8000} = -\frac{31}{8000} < 0.
\]
Similarly, we can derive an $H^1(\T^3)$ estimate: \\
\begin{align*}
\Honenorm{\omp_{n+1}} &= \Honenorm{\sum_{k}a_k \WW^\npone_k(\sigma_{n+1}\cdot)}   \\\\ 
&\les \twonorm{\sum_{k} \nabla a_k \WW^\npone_k(\sigma_{n+1}\cdot)} + \twonorm{\sum_{k} a_k \nabla \WW^\npone_k(\sigma_{n+1}\cdot)} \\ \notag \\ 
&\les \Conenorm{a_k} \twonorm{\WW^\npone_k} + \sigma_{n+1}\twonorm{a_k} \twonorm{\nabla \WW^\npone_k}+ \sigma^{\haf}_\npone \Conenorm{a_k^2} \twonorm{\nabla \WW^\npone_k}  \\\\
&\les l^{-9} \d^\onehalf_{n+1} + \sigma_{n+1} \mu_{n+1} \d^\onehalf_{n+1} + \sigma^\haf_\npone l^{-9}\mu_\npone \ddelta^\haf \les \sigma_{n+1} \mu_{n+1} \d^\onehalf_{n+1}.
\end{align*}\\
Remember that $\l_{n+1} = \sigma_{n+1} \mu_{n+1}$ and $l^{-9}= \l_{n+1}^{9\gm} << \l_{n+1}$, therefore \\
\begin{equation}\label{H1_principle}
 \Honenorm{\omp_{n+1}}  \les \l_{n+1} \d^\onehalf_{n+1}. 
\end{equation}\\
Sobolev interpolation yields the $H^s(\T^3)$ bound: \\
\begin{equation}\label{Hs_principle}
\Hsnorm{\omp_{n+1}} \les \twonorm{\omp_{n+1}}^{1-s}  \Honenorm{\omp_{n+1}}^s \les \l^s_{n+1} \d^\onehalf_{n+1},
\end{equation}\\
for $0 \leq s \leq 1$. Additionally, we can deduce the $L^1(\T^3)$ boundedness for $\omp_{n+1}$ via lemma \ref{lemma_immproved_holder} : 
\begin{align}\notag
\onenorm{\omp_{n+1}} &\les \onenorm{a_k} \onenorm{\WW^\npone_k} + \sigma^{-1}_{n+1}\Conenorm{a_k} \onenorm{\WW^\npone_k} \\ \notag \\ \notag 
& \les \twonorm{a_k}\onenorm{\WW^\npone_k} + \sigma^{-1}_{n+1}\Conenorm{a_k} \onenorm{\WW^\npone_k}\\ \notag \\ \label{L1_principle}
&\les \mu^{-1}_{n+1} \d^\onehalf_{n+1} + l^{-9} \sigma^{-1}_{n+1} \mu^{-1}_{n+1}  \d^\onehalf_{n+1} \les \mu^{-1}_{n+1} \d^\onehalf_{n+1}. 
\end{align}
This is valid if we have $\sigma^{-1}_{n+1}l^{-9} \les 1$, which is true since we have shown that $\sigma^{-\haf}_{n+1}l^{-9} \les 1$ before.\\\\
By H\"older interpolation, we also have that for $1<r<2$, \\
\begin{equation}\label{Lr_principle}
\rnorm{\omp_{n+1}} \les \onenorm{\omp_{n+1}}^{\frac{2-r}{r}} \twonorm{\omp_{n+1}}^{\frac{2r-2}{r}} \les \mu_{n+1}^{-\frac{2-r}{r}}\ddelta^\onehalf.
\end{equation}\\
Likewise, we can also derive the following bounds for $\omc_{n+1}$. Recall that 
\[
\omc_{n+1} = \sigma^{-1}_{n+1}\sum_{k \in \lb} \nabla a_k(x) :\OOm^\npone_k(\sigma_{n+1}x),
\]
we therefore have\\\\
\textit{$L^2$ boundedness}:\\
\begin{align}
\twonorm{\omc_{n+1}} &\les \sigma^{-1}_{n+1} \Linfnorm{\nabla a_k} \twonorm{\OOm^\npone_k}\notag \\
&\les \sigma^{-1}_{n+1} \mu^{-1}_{n+1} l_{n+1}^{-9} \d^\onehalf_{n+1} \notag  \\
&\les \l^{9\gm -1}_{n+1}  \d^\onehalf_{n+1}, \label{L2_corrector}
\end{align}\\
\textit{$H^1$ boundedness}:\\
\begin{align*}
\Honenorm{\omc_{n+1}} &= \sigma^{-1}_{n+1}\HoneNorm{\sum_{k} \nabla a_k : \OOm^\npone_k(\sigma_{n+1}\cdot)}  \\ 
&\les \sigma^{-1}_{n+1} \Conenorm{a_k} \sigma_{n+1} \twonorm{\nabla \OOm^\npone_k} +\sigma^{-1}_{n+1}\Ctwonorm{a_k}\twonorm{\OOm^\npone_k}\\\\
&\les l^{-9} \d^\haf_{n+1} + \sigma^{-1}_{n+1} l^{-10}\d^\haf_{n+1} \mu^{-1}_{n+1}  \les l^{-9} \d^\haf_{n+1} + \l^{-1}_{n+1} l^{-10}  \d^\haf_{n+1}
\end{align*}\\
\begin{equation}\label{H1_corrector}
  \Longrightarrow \Honenorm{\omc_{n+1}} \les \Big ( \l^{9\gm}_{n+1}+ \l^{-1+10\gm}_{n+1} \Big) \d^\haf_{n+1} \les \l_{n+1}\d^\haf_{n+1}. 
\end{equation}\\
Hence by interpolation we secure the $H^s$ boundedness: \\
\begin{equation}\label{Hs_corrector}
\Hsnorm{\omc_{n+1}} \les \l^s_{n+1} \d^\onehalf_{n+1}, \ \ \ s \leq 1.
\end{equation}\\
\textit{$L^1$ boundedness}:\\
\begin{align*}
\onenorm{\omc_{n+1}} &\les \sigma^{-1}_{n+1} \Conenorm{a_k} \onenorm{\OOm^\npone_k} \les\sigma^{-1}_{n+1}l^{-9}  \mu^{-2}_{n+1}\d^\onehalf_{n+1} \\\\
&\les \l^{-1+9\gm}_{n+1} \powermone{\mu_{n+1}} \d^\onehalf_{n+1} \les \powermone{\mu_{n+1}}\ddelta^\onehalf.
\end{align*}\\
 And again by interpolation, we have for $1 < r < 2$,\\
\begin{equation}\label{Lr_corrector}
\rnorm{\omc_{n+1}} \les \mu_{n+1}^{-\frac{2-r}{r}}\ddelta^\onehalf.
\end{equation}\\
It is worth noticing that \\
\begin{equation}\label{compare_pandc}
\omc_{n+1} \ll \omp_{n+1} 
\end{equation}\\
in both $L^r(\T^3)$ and $H^{s}(\T^3)$. We summarize the bounds of $\om_{n+1}$ in the following proposition:\\
\begin{Proposition}\label{perturbation_estimation} The magnetic perturbation $\om_{n+1}$ has the following bounds:\\
\begin{enumerate}[label=\slshape(\roman*)]
\item \label{Hs_perturbation}
For $0 \leq s \leq 1$, we have the $H^s$ bound: \\
\begin{equation*}
\Hsnorm{\om_{n+1}} \les \l^s_{n+1} \ddelta^\onehalf.
\end{equation*}\\
\item \label{Lr_perturbation} For $r \in (1,2)$, we have the $L^r$ bound: \\
\begin{equation*}
\rnorm{\om_{n+1}} \les \mu_{n+1}^{-\frac{2-r}{r}}\ddelta^\onehalf.
\end{equation*}\\
\end{enumerate}
\end{Proposition}
\begin{proof}
The first one of the above follows from $\eqref{Hs_principle}$ and $\eqref{Hs_corrector}$ while the second one follows from $\eqref{Lr_principle}$ and $\eqref{Lr_corrector}$.\\
\end{proof}
A direct consequence of proposition $\ref{perturbation_estimation}$ is that $(\ref{Inductive Hypothesis 2})$ and $(\ref{Inductive Hypothesis 3})$ hold for $B_{n+1}$ with $n$ replaced by $n+1$. In the following sections, we will develop an $L^r$ bound for $R_{n+1}$ for some $1<r<2$ such that $r$ is very close to 1. 
\subsubsection{Estimate the linear error $R_L$} Recall from $(\ref{R_L})$ that \\
\begin{equation*}
R_L =  \Rr \Big( \nabla \times \om_{n+1} + \divv \big( \mB_l \otimes \om_{n+1} + \om_{n+1} \otimes \mB_l \big) \Big). 
\end{equation*}\\
We use the fact that the operator $(\Rr \nabla \times )$ is bounded from $L^r$ to $L^r$ for $1 < r < \infty$ as well as the Sobolev embedding $  H^2 \hookrightarrow L^\infty$ to deduce that \\
\begin{align*}
\rnorm{R_L} &\les \rnorm{\Rr \nabla \times \om_{n+1}} + \rnorm{\mB_l \otimes \om_{n+1}}+ \rnorm{\om_{n+1} \otimes  \mB_l }\\\\
&\les \rnorm{\om_{n+1}} + \Linfnorm{\mB_l} \rnorm{\om_{n+1}} \\\\
&\les \rnorm{\om_{n+1}} + ||\mB_l ||_{H^2(\T^3)} \rnorm{\om_{n+1}} \les ||\mB_l ||_{H^2(\T^3)} \rnorm{\om_{n+1}}.
\end{align*}\\
Furthermore, applying Lemma $\ref{mollification_lemma}$ and Lemma $\ref{perturbation_estimation}$,\\
\begin{equation}\label{rnorm_linear}
\rnorm{R_L} \les l^{-1}\l_n \d^\onehalf_n \cdot \mu_{n+1}^{-\frac{2-r}{r}}\ddelta^\onehalf \les \l^{\frac{1-\beta}{b}+\gm}_{n+1}\l_{n+1}^{(1-\a)(1-\frac{2}{r})} \ddelta^\haf \les \d_{n+2}.
\end{equation}\\
The last inequality is satisfied provided that \\
\begin{equation*}
2\beta b +\frac{2(1-\b)}{b} + (1-\a)(1-\frac{2}{r}) < 0. 
\end{equation*}\\
Inserting the number $\b = \frac{1}{250}$, $b = 32$ and $\a=\frac{16}{25}$, we obtain
\[
2\beta b +\frac{2(1-\b)}{b} + (1-\a)(1-\frac{2}{r}) \approx -\frac{167}{4000} < 0.
\]
if we choose $r$ close enough to 1. 
\vspace{0.5cm}
\subsubsection{Estimate the corrector error} Recall from $(\ref{R_C})$ that \\
\begin{equation*}
R_C = \Rr \Big ( \divv \big( \omc_{n+1} \otimes \om_{n+1} + \omp_{n+1} \otimes \omc_{n+1} \big )  \Big), 
\end{equation*}\\
by Sobolev embedding $H^{3(1-\frac{1}{r})}(\T^3) \hookrightarrow L^{\frac{2r}{2-r}}(\T^3)$ and $(\ref{compare_pandc})$\\
\begin{align*}
\rnorm{R_C} &\les \rnorm{\omc_{n+1}\otimes \om_{n+1}} + \rnorm{\omp_{n+1}\otimes \omc_{n+1}} \\\\
&\les \twonorm{\omc_{n+1}} \Big( ||\om_{n+1}||_{L^{\frac{2r}{2-r}}} + ||\omp_{n+1}||_{L^\frac{2r}{2-r}} \Big)\\\\
&\les \twonorm{\omc_{n+1}} ||\omp_{n+1}||_{L^\frac{2r}{2-r}} \les \twonorm{\omc_{n+1}} ||\omp_{n+1}||_{H^{3(1-\frac{1}{r})}}.
\end{align*}\\
Utilize $(\ref{Hs_principle})$ and $(\ref{L2_corrector})$ to obtain \\
\begin{equation}\label{rnorm_corrector}
\rnorm{R_C} \les \l_{n+1}^{-1+9\gm}\l_{n+1}^{3(1-\frac{1}{r})} \ddelta \les \d_{n+2}.
\end{equation}\\
The last inequality holds since  
\begin{equation*}
 2\beta (b-1) + 9\gm -1 + 3\Big( 1-\frac{1}{r} \Big) \approx -0.43 < 0,
\end{equation*}
for $r$ sufficiently close to $1$.
\subsubsection{Estimate the oscillation error $R_{osc}$} Recall from $(\ref{R_oscx})$ that \\
\begin{equation*}
R_{\textit{osc}} = \sum_{k \in \lb} \Tt \Big( \nabla(a^2_k), \WW^\npone_k(\sigma_{n+1}x)\otimes \WW^\npone_k(\sigma_{n+1}x) - \fint_{\T^3} \WW^\npone_k \otimes \WW^\npone_k \Big).
\end{equation*}\\
For the convenience of presentation let us define that 
\begin{equation*}
\W_k(x) =\WW^\npone_k(x) \otimes \WW^\npone_k(x) - \fint_{\T^3} \WW^\npone_k \otimes \WW^\npone_k.
\end{equation*}\\
Applying Theorem \ref{Theorem_bilinear_antidiv} we deduce that \\
\begin{align*}
\rnorm{R_{osc}} &\les \sum_k \rnorm{\Tt(\nabla(a_k^2), \W_k)}\\
&\les \sum_{k} \Conenorm{\nabla(a^2_k)} \rnorm{\Rr \Big(\W_k(\sigma_{n+1}\cdot) \Big)} \\\\
&\les \Ctwonorm{a^2_k} \sigma_{n+1}^{-1} \rnorm{\W_k} \les \sigma_{n+1}^{-1}\Ctwonorm{a^2_k} \rnorm{\WW^\npone_k \otimes \WW^\npone_k}  \\\\
&\les  \sigma_{n+1}^{-1}\Ctwonorm{a^2_k} ||\WW^\npone_k||^2_{L^{2r}(\T^3)}. 
\end{align*}\\
Thanks to lemma $\ref{lemma_Amplitude_estimation}$ and Theorem $\ref{prop_Mikado}$, \ref{prop_Mikado2}, we have that\\
\begin{equation}\label{rnorm_oscx}
 \rnorm{R_{osc}} \les \sigma_{n+1}^{-1} l^{-10}  \d_{n+1} \mu_{n+1}^{2-\frac{2}{r}} \les \d_{n+2},
\end{equation}\\
where the last inequality is ensured by \\
\begin{equation*}
 0 > 2\beta(b-1) + 10\gm -\a + 2 (1-\a) \Big( 1-\frac{1}{r} \Big) \approx -0.03
\end{equation*}
with $r$ being close enough to 1.
\subsubsection{Estimate the high frequency interference $R_F$} Recall from $(\ref{R_F})$ that \\
\begin{equation*}
R_F  = \sum_{k \neq k'} a_k a_{k'} \WW^\npone_k(\sigma_{n+1}\cdot) \otimes \WW^\npone_{k'}(\sigma_{n+1}\cdot),
\end{equation*}\\
hence by lemma $\ref{lemma_immproved_holder}$ 
\begin{align*}
\rnorm{R_F} &\les \rnorm{\sum_{k \neq k'}a_k a_{k'} \WW^\npone_k(\sigma_{n+1}\cdot) \otimes \WW^\npone_{k'}(\sigma_{n+1}\cdot)}\\\\
&\les \rnorm{a^2_k} \rnorm{\WW^\npone_k \otimes \WW^\npone_{k'}} + \sigma^{-\frac{1}{r}}_\npone \Conenorm{a^2_k} \rnorm{\WW^\npone_k \otimes \WW^\npone_{k'}} \\\\
&\les \Big( \twonorm{a_k}^{\frac{2}{r}} \|a_k\|^{\frac{2r-2}{r}}_{L^\infty(\T^3)} + \sigma_\npone^{-\frac{1}{r}}\Conenorm{a^2_k}\Big) \rnorm{\WW^\npone_k \otimes \WW^\npone_{k'}} \\\\
&\les \Big(  l^{\frac{8(r-1)}{r}} \ddelta + \sigma_\npone^{-\frac{1}{r}} l^{-9} \ddelta \Big) \rnorm{\WW^\npone_k \otimes \WW^\npone_{k'}} \\\\
&\les \sigma_\npone^{-\frac{1}{r}} l^{-9} \ddelta \rnorm{\WW^\npone_k \otimes \WW^\npone_{k'}}.
\end{align*}\\
Finally applying property \ref{prop_Mikado3} in Theorem $\ref{prop_Mikado}$ yields \\
\begin{equation}\label{rnorm_interference}
\rnorm{R_F} \les l^{-9}\sigma_\npone^{-\a} \mu^{2-\frac{3}{r}}_{n+1}\ddelta \les \l^{9\gm -\a + (1-\a)(2-\frac{3}{r})-2\beta}_{n+1}
\les \d_{n+2}, 
\end{equation}\\
which is implied by 
\begin{equation*}
0 > 2\beta (b-1) + 9\gm-\a +(1-\a) \Big(2-\frac{3}{r} \Big) \approx -0.43.
\end{equation*}\\
With $(\ref{rnorm_linear})$, $(\ref{rnorm_corrector})$, $(\ref{rnorm_oscx})$ and $(\ref{rnorm_interference})$, it becomes apparent that \\
\begin{align*}
\onenorm{R_{n+1}} &\les \rnorm{R_{n+1}} \\\\
&\les \rnorm{R_L}+\rnorm{R_C}+\rnorm{R_{osc}}+\rnorm{R_F} \\\\
&\les \d_{n+2},
\end{align*}\\
which is $(\ref{Inductive Hypothesis 1})$ with $n$ replaced by $n+1$. 
\qed

\section*{Acknowledgments} 
The author is grateful to Professor Alexey Cheskidov for the open problem and multiple inspirational discussions. Special thanks go to Professor Mimi Dai for providing the UIC summer research fellowship during which this paper was completed. 

\newpage
\appendix
\section{Technical tools}\label{appendix1}
This section includes a collection of propositions and lemmas that we used in this paper. The first one is the commutator estimate initially done by Constantin, Eyink and Titi\cite{CET94}. 
\begin{Proposition}\label{lemma_CET_commutator_est}
Let $f,g$ be smooth functions on $\T^3$ and $\eta_l$ the standard mollifier with length scale $l$, then for any $m \in \N$ and $r \in [1,\infty]$, they satisfy the following estimate 
\begin{equation}\label{lemma_CET_commutator_est1}
\rnorm{\nabla^m \Big( (fg)\ast \eta_l -(f\ast \eta_l)(g \ast \eta_l) \Big )} \les l^{2-m} \| \nabla f\|_{L^{2r}(\T^3)} \| \nabla g\|_{L^{2r}(\T^3)}.
\end{equation}
\end{Proposition}
\begin{proof}
See \cite{L19}, proposition B.1.
\end{proof}
The next one is known as the improved H\"older's inequality which is due to Modena and Székelyhidi\cite{MS18}. This result provides an improved estimate when one of the two multiplying functions has much higher oscillation. 
\begin{Lemma}\label{lemma_immproved_holder}
Suppose $f,g$ are two smooth functions on $\T^3$ and $r \in [1,\infty]$, then for any $\sigma \in \N$ we have 
\begin{equation}\label{lemma_improved_holder1}
\Big|\rnorm{fg(\sigma \cdot)} - \rnorm{f}\rnorm{g} \Big | \les \sigma^{-\frac{1}{r}} \Conenorm{f} \rnorm{g}.
\end{equation}
\end{Lemma}
\begin{proof}
See \cite{MS18}, lemma 2.1.
\end{proof}
The following results provide a H\"older estimate for multivariate function compositions. 
\begin{Proposition}\label{lemma_holder_composition}
Let $\Psi: \Omega \to \R$ and $u: \R^n \to \Om$ be two smooth functions with $\Om \subset \R^d$. The for any positive integer $m$, we have the following estimate
\begin{equation}\label{lemma_holder_composition1}
\|\nabla^m(\Psi \circ u) \|_{L^\infty} \les \|\nabla^m u\|_{L^\infty} \sum_{i=1}^m \|\nabla^i \Psi\|_{L^\infty} \|u \|^{i-1}_{L^\infty}.
\end{equation}
\end{Proposition}
\begin{proof}
See \cite{DLS14}, proposition 4.1. 
\end{proof}

\section{The Antidivergence operator}\label{appendix2}
This section provides a fundamental background of the antidivergence operator $\Rr$ and the bilinear antidivergence operator $\Tt$ introduced in \cite{DLS13} and \cite{CL22} respectively. 
\begin{Definition}[\cite{CL22}, definition B.2]\label{def_antidiv} 
Denote $\Ss^{3\times 3}_0$ the set of traceless $3 \times 3$ matrices. The antidivergence operator $\Rr: C^\infty(\T^3;\R^3) \to C^\infty (\T^3; \Ss^{3\times 3}_0)$ is defined by 
\begin{equation}\label{def_antidiv_formula1}
(\Rr v)_{ij} = \Rr_{ijk} v_k,
\end{equation}
in which 
\begin{equation}\label{def_antdiv_formula2}
\Rr_{ijk} = -\frac{2-d}{d-1} \Delta^{-2} \p_i \p_j \p_k - \frac{1}{d-1} \Delta^\mone \p_k \d_{ij} + \Delta^\mone \p_i \d_{jk} + \Delta^\mone \p_j \d_{ik}.
\end{equation}
\end{Definition}

\begin{Lemma}[\cite{DLS13}, lemma 4.3]\label{lemma_antidiv_properties} 
The antidivergence operator $\Rr$ satisfies the following properties: \\
\begin{enumerate}[label=\itshape(\roman*)]
\item $\Rr$ is a traceless symmetric matrix on $\T^3$.\\
\item $\Rr$ serves as an inverse operator of the divergence, i.e.
\[
\divv (\Rr v ) = v - \fint_{\T^3} v.  
\]
\end{enumerate}
\end{Lemma}
\begin{proof}
Direct computation. For detail see \cite{CL22}, Appendix B.2. 
\end{proof}

\begin{Theorem}\label{Theorem_antidiv_Lp_bound}
Let $u \in C^\infty(\T^3)$ and $1 \leq r \leq \infty$, one has 
\begin{equation*}
\rnorm{\Rr u} \lesssim \rnorm{u}.
\end{equation*}
In particular if $u \in C_0^\infty (\T^3)$, then for any $\sigma \in \N$ we have 
\[
\rnorm{\Rr u(\sigma \cdot)} \les \sigma^\mone \rnorm{u}.
\]
\end{Theorem}
\begin{proof}
See \cite{CL22}, Theorem B.3.
\end{proof}

\begin{Definition}[\cite{CL22}, Appendix B.3]\label{def_bilinear_antidiv}
The Bilinear antidivergence operator $\Tt : C^\infty(\T^3,\R^3) \times C^\infty(\T^3, \R^{3 \times 3}) \to C^\infty(\T^3, \Ss_0^{3 \times 3})$ is defined as 
\begin{equation}\label{def_bilinear_antidiv1}
\big(\Tt(u,H) \big)_{ij} := u_l \Rr_{ijk}H_{lk} - \Rr (\p_i u_l \Rr_{ijk}H_{lk}),
\end{equation}
or 
\[
\Tt(u,H) := u \Rr H - \Rr (\nabla u \Rr H).
\]
\end{Definition}

\begin{Theorem}\label{Theorem_bilinear_antidiv}
Given $u \in C^\infty(\T^3, \R^3)$ and $H \in C_0^\infty(\T^3, \R^{3 \times 3})$, there holds
\begin{equation}\label{Theorem_bilinear_antidiv1}
\divv \big(\Tt (u,H) \big) = uH - \fint_{\T^3} uH.
\end{equation}\\
In addition, for any $1 \leq r \leq \infty $, 
\begin{equation}\label{Theorem_bilinear_antidiv2}
\rnorm{\Tt (u,H)} \les \Conenorm{u} \rnorm{\Rr H}.
\end{equation}
\end{Theorem}

\begin{proof}
See \cite{CL22}, Theorem B.4.
\end{proof}

\bibliographystyle{plain}
\bibliography{Stat_EMHD}
\end{document}